\documentclass[12pt]{article}

\usepackage{tikz}
\usepackage{subfigure}
\usepackage[english]{babel}

\usepackage[center]{caption2}
\usepackage{amsfonts,amssymb,amsmath,latexsym,amsthm}
\usepackage{multirow}

\def\gm{\gamma}
\def\ex{\mathrm{ex}}
\def\ecrossG{\sum_{1\leqslant i<j \leqslant r-1}e(V_i,V_j)}
\def\ecrossT{\sum_{1\leqslant i<j \leqslant r-1}|V_i||V_j|}
\def\ecrossTT{\sum_{2\leqslant i<j \leqslant r-1}|V_i||V_j|}
\def\einG{\sum_{i=1}^{r-1}e(V_i)}
\def\aver{\frac{n}{r-1}}
\def\error{2\sqrt{\gm} n}
\def\lowbound{\aver-2\sqrt{\gm} n}
\def\upbound{\aver+2\sqrt{\gm} n}
\def\lf{\left\lfloor}
\def\rf{\right\rfloor}
\def\lc{\left\lceil}
\def\rc{\right\rceil}
\def\dtaT{\lf\frac{r-2}{r-1}n\rf}
\def\bad{32kr^3\gm n}
\def\inactive{8kr^5\gm n}
\def\Fkr{F_{k,r}}
\def\geqs{\geqslant}
\def\leqs{\leqslant}

\def\sgmn{\sqrt{\gm}n}
\def\gmnn{\gm n^2}

\newtheorem{thm}{Theorem}

\newtheorem{lem}[thm]{Lemma}

\newtheorem{conj}[thm]{Conjecture}

\newtheorem{claim}{Claim}

\newtheorem{case}{Case}

\begin{document}

\title{Decomposition of Graphs into $(k,r)$-Fans and Single Edges
\thanks{The work was supported by NNSF of China (No. 11271348).}}
\author{Xinmin Hou$^a$, \quad Yu Qiu$^b$, \quad Boyuan Liu$^c$\\
\small$^{a,b,c}$ Key Laboratory of Wu Wen-Tsun Mathematics\\
\small Chinese Academy of Sciences\\
\small School of Mathematical Sciences\\
\small University of Science and Technology of China\\
\small Hefei, Anhui 230026, China.\\
\small $^a$xmhou@ustc.edu.cn,\quad  $^b$yuqiu@mail.ustc.edu.cn\\
\small $^c$lby1055@mail.ustc.edu.cn}

\date{}

\maketitle

\begin{abstract}
Let $\phi(n,H)$ be the largest integer such that, for all graphs $G$ on $n$ vertices, the edge set $E(G)$ can be partitioned into at most $\phi(n, H)$ parts, of which every part either is a single edge or forms a graph isomorphic to $H$. Pikhurko and Sousa conjectured that $\phi(n,H)=\ex(n,H)$ for $\chi(H)\geqs3$ and all sufficiently large $n$, where $\ex(n,H)$ denotes the maximum number of edges of graphs on $n$ vertices that does not contain $H$ as a subgraph. A $(k,r)$-fan is a graph on $(r-1)k+1$ vertices consisting of $k$ cliques of order $r$ which intersect in exactly one common vertex. In this paper, we verify Pikhurko and Sousa's conjecture for $(k,r)$-fans. The result also generalizes a result of Liu and Sousa.
\end{abstract}

\section{Introduction}

All graphs considered in this paper are simple and finite. Given a graph $G=(V, E)$ and a vertex $x\in V(G)$, the number of neighbors of $x$ in $G$, denoted by $\deg_G(x)$, is called the {\em degree} of $x$ in $G$. The number of edges of $G$ is denoted by $e(G)$. A {\em{matching}} in $G$ is a subgraph of $G$ such that each of its vertices has degree 1. The {\em matching number} of $G$ is the maximum number of edges in a matching of $G$, denoted by $\nu(G)$. As usual, we use $\delta(G),\Delta(G)$ and $\chi(G)$ to denote the minimum degree, maximum degree, and chromatic number of $G$, respectively. For a graph $G$ and $S, T\subset V(G)$, let $e_G(S, T)$ be the number of edges $e=xy\in{E(G)}$ such that $x\in S$ and $y\in T$, if $S=T$, we use $e_G(S)$ instead of $e_G(S, S)$; and  we use $e_G(u, T)$ instead of $e_G(\{u\}, T)$ for convenience, the index $G$ will be omitted if no confusion from the context. For a subset $X\subseteq V(G)$ or $X\subseteq  E(G)$, let $G[S]$ be the subgraph of $G$ induced by $X$, that is $G[X]=(X, E(X))$ if $X\subseteq V(G)$, or $G[X]=(V(X),X)$ if $X\subseteq  E(G)$. For two integers $a, b$ with $a\le b$, let $[a,b]=\{a, a+1,\cdots, b\}$.

Let $K_r$ denote the complete graph of order $r$ and let $T_{n,r}$ denote the complete balanced $r$-partite graph of order $n$, also called the {\em Tur\'an graph} in literature. For $k\geqs2$ and $r\geqs3$, a {\em $(k,r)$-fan}, denoted by $F_{k,r}$, is the graph on $(r-1)k+1$ vertices consisting of $k$ $K_r$'s which intersect in exactly one common vertex, called the center of it.
For some fixed graph $H$, let $\ex(n,H)$ be the maximum number of edges of graphs on $n$ vertices that does not contain $H$ as a subgraph. A graph $G$ is called an {\em extremal graph for $H$} if $G$ has $n$ vertices with $e(G)=\ex(n, H)$ and does not contain $H$ as a subgraph.
Given two graphs $G$ and $H$, an {\em $H$-decomposition} of $G$ is a partition of edges of $G$ such that every part is a single edge or forms a graph isomorphic to $H$. Let $\phi(G,H)$ be the smallest number of parts in an $H$-decomposition of $G$. Clearly, if $H$ is non-empty, then
$$\phi(G,H)=e(G)-p_{H}(G)(e(H)-1),$$
where $p_{H}(G)$ is the maximum number of edge-disjoint copies of $H$ in $G$. Define
$$\phi(n,H)=\max\{\phi(G,H):\ G \mbox{ is a graph on $n$ vertices}\}.$$

This function, motivated by the problem of representing graphs by set intersections, was first studied by Erd\"{o}s, Goodman and P\'{o}sa \cite{intersections}, they proved that $\phi(n,K_3)=\ex(n,K_3)$. The result was generalized by Bollob\'{a}s \cite{Bollobas}, he proved that $\phi(n,K_r)=\ex(n,K_r)$, for all $n\geqs{r}\geqs3$.
More generally,  Pikhurko and Sousa~\cite{general} proposed the following conjecture.
\begin{conj}[\cite{general}]\label{CONJ: Pikhurko and Sousa}
For any graph $H$ with $\chi(H)\geqs3$, there is an $n_0=n_0(H)$ such that $\phi(n,H)=\ex(n,H)$ for all $n\geqs{n_0}$.
\end{conj}
In~\cite{general}, Pikhurko and Sousa also proved that $\phi(n,H)=\ex(n,H)+o(n^2).$ Recently, the error term improved to be $O(n^{2-\alpha})$ for some $\alpha>0$ by Allen, B\"{o}ttcher, and Person~\cite{ImprovedError}. Sousa verified the conjecture for some families of edge-critical graphs, namely, clique-extensions of order $r\geqs4$ $(n\geqs{r})$~\cite{clique-extension} and the cycles of length 5 $(n\geqs6)$~\cite{5-cycles} and 7 $(n\geqs10)$~\cite{7-cycle}.
In~\cite{edge-critical case}, \"{O}zkahya and Person verified the conjecture for all edge-critical graphs with chromatic number $r\geqs3$. Here, a graph $H$ is called {\em edge-critical}, if there is an edge $e\in{E(H)}$, such that $\chi(H)>\chi(H-e)$. For non-edge-critical graphs, Liu and Sousa~\cite{k-fan} verified the conjecture for $(k,3)$-fans, there result is as the following.

\begin{thm}[\cite{k-fan}]\label{THM: Liu-Sousa} For $k\geqs1$, there exists $n_0=n_0(k,r)$ such that $\phi(n,F_{k,3})=\ex(n,F_{k,3})$ for all $n\geqs{n_0}$. Moreover, the only graphs attaining $\ex(n,F_{k,3})$ are the extremal graphs for $F_{k,3}$.
\end{thm}

In this paper, we verify Conjecture~\ref{CONJ: Pikhurko and Sousa} for $(k,r)$-fans for $k\ge 2$ and $r\ge 3$ and hence generalizes Theorem~\ref{THM: Liu-Sousa}.  Our main result is the following.

\begin{thm}\label{THM: Main}
For $k\geqs 2$ and $r\geqs3$, there exists $n_1=n_1(k,r)$ such that $\phi(n,\Fkr)=\ex(n,\Fkr)$ for all $n\geqs{n_0}$. Moreover, the only graphs attaining $\ex(n,\Fkr)$ are the extremal graphs for $\Fkr$.
\end{thm}

The remaining of the paper is arranged as follows. Section 2 gives some lemmas. The proof of Theorem~\ref{THM: Main} is given in Section 3.

\section{Lemmas}
The extremal graphs for $\Fkr$ was determined by  Chen, Gould, Pfender and Wei~\cite{Wei}.
\begin{lem}[Theorem 2 in~\cite{Wei}]\label{THM:Fkr}
For every $k\geqslant1$ and $r\geqslant2$ and every $n\geqslant16k^3r^8$, $\ex(n,F_{k,r})=\ex(n,K_r)+g(k)$, where $$g(k)=\left\{
\begin{aligned}
&k^2-k & &{if~k~is~odd,}\\
&k^2-\frac{3}{2}k & &{if~k~is~even.}\\
\end{aligned} \right.$$
And one of its extremal graphs, denoted by $G_{n,k,r}$, constructed as follows. If $k$ is odd, $G_{n,k,r}$ is a Tur\'an graph $T_{n,r-1}$ with two vertex disjoint copies of $K_k$ embedding in one partite set. If $k$ is even, $G_{n,k,r}$ is a Tur\'an graph $T_{n,r-1}$ with a graph on $2k-1$ vertices, $k^2-\frac{3}{2}k$ edges, and maximum degree $k-1$ embedded in one partite set.
\end{lem}

\begin{lem}\label{PROP: ex-p1}
Let $G$ be an extremal graph on $n$ vertices for $\Fkr$. Then $\delta(G)\ge \lf\frac{r-2}{r-1}n\rf$.
\end{lem}
\begin{proof}
Suppose to the contrary that there is a vertex $v\in V(G)$ with $\deg_G(v)<\dtaT=\delta(T_{n,r-1})$. Let $G'= G-v$. Then $$e(G')\geqs e(G)-\deg_G(v)\ge \ex(n,\Fkr)-\delta(T_{n,r-1})+1=\ex(n-1,F_{k,r})+1$$ since $\ex(n,F_{k,r})-\ex(n-1,F_{k,r})=\delta(T_{n,r-1})$. By Lemma~\ref{THM:Fkr}, $G'$ (and hence $G$) contains a copy of $\Fkr$ as its subgraph, a contradiction.
\end{proof}

\begin{lem}\label{Lemma:deltaGgeq}
 Let $n_0$ be an integer and let $G$ be a graph on $n\ge n_0+{n_0\choose 2}$ vertices with $\phi(G,\Fkr)=\ex(n,\Fkr)+j$ for some integer $j>0$. Then $G$ contains a subgraph $G'$ on $n'> n_0$ vertices  such that $\delta(G')\geqs\delta(T_{n-i,r-1})$ and $\phi(G',\Fkr)\geqs\ex(n-i,\Fkr)+j+i$.
\end{lem}
\begin{proof}
If $\delta(G)\ge \lf\frac{n}2\rf$, then $G$ is the desired graph and we have nothing to do. So assume that $\delta(G)<\lf\frac{n}2\rf$.
Let $v\in V(G)$ with $\deg_G(v)<\dtaT$ and set $G_1= G-v$. Then $\phi(G_1,\Fkr)\geqs\phi(G,\Fkr)-\deg_G(v)\geqs \ex(n,\Fkr)+j-\delta(T_{n,r-1})+1=\ex(n-1,F_{k,r})+j+1$, since $\ex(n,F_{k,r})-\ex(n-1,F_{k,r})=\delta(T_{n,r-1})$ by Lemma~\ref{THM:Fkr}. We may continue this procedure until we get a graph $G'$ on $n-i$ vertices with $\delta(G')\geqs\lf\frac{r-2}{r-1}(n-i)\rf$ for some $i<n-n_0$, or until $i=n-n_0$. But the latter case can not occur since $G'$ is a graph on $n_0$ vertices but $e(G')\geqs\ex(n_0,F_{k,r})+j+i\geqs n-n_0>\binom{n_0}{2}$, which is impossible.
\end{proof}


The following stability lemma due to \"{O}zkahya and Person~\cite{edge-critical case} is very important. 

\begin{lem}[\cite{edge-critical case}]\label{LEM: partition} Let $H$ be a graph with $\chi(H)=r\geqs3$ and $H\neq{K_r}$. Then, for every $\gm>0$ there exists $\delta>0$ and $n_0=n_0(H,\gm)\in\mathbb{N}$ such that for every graph $G$ on $n\geqs{n_0}$ vertices with $\phi(G, H)\geqs\ex(n,H)-\delta{n^2}$, then there exists a partition of $V(G)=V_1\dot{\cup}...\dot{\cup}V_{r-1}$ such that $\sum^{r-1}_{i=1}e(V_i)<\gm n^2$.
\end{lem}


\begin{lem}[\cite{Hanson}]\label{NuDelta} For any graph $G$ with maximum degree $\Delta\geqs1$ and matching number $\nu\geqs1$, then $e(G)\leqslant f(\nu,\Delta)$. Here, $f(\nu,\Delta)=\nu\Delta+{\lf\frac{\Delta}{2}\rf}{\lf\frac{\nu}{\lc\Delta/2\rc}\rf}$.
\end{lem}
\section{Proof of Theorem \ref{THM: Main}}

The lower bound $\phi(n,H)\geqs\ex(n,H)$ is trivial by the definition of $\phi(n, H)$ and $\ex(n, H)$. To prove Theorem \ref{THM: Main}, it is sufficient to prove that $\phi(n,\Fkr)\leqs\ex(n,\Fkr)$ and the equality holds only for extremal graphs for $\Fkr$. Our proof is motivated by the one in~\cite{k-fan}. In outline, for every graph $G$ with $\phi(G,\Fkr)\geqs\ex(n,\Fkr)$ and $G$ is not an extremal graph for $F_{k,r}$, we will find sufficiently many edge-disjoint copies of $\Fkr$ in $G$ which would imply that $\phi(G,\Fkr)=e(G)-p_{\Fkr}(G)(e(\Fkr)-1)<\ex(n,\Fkr)$, a contradiction. In other words, we will prove that $$p_{\Fkr}(G)>\frac{e(G)-\ex(n,\Fkr)}{e(\Fkr)-1}.$$

For convenience, we set some constants as follows.
\begin{eqnarray*}
&&\gm=(40kr^4)^{-2},\\
&&n_0=n_0(\Fkr,\gm)\  (\mbox{ which is determined by $\Fkr$ and $\gm$ by applying Lemma~\ref{LEM: partition}}),\\
&&\alpha=\sqrt{1-\frac{r-1}{r-2}\gm},\\
&&m_1=m_1(k,r)=\frac{(e(\Fkr)-1)(r-1)k(k-1)-kg(k)}{e(\Fkr)-1-k},\\
&&m_2=m_2(k,r)=\frac{(e(\Fkr)-1)(r-1)k(k-1)-kg(k)}{(e(\Fkr)-1)/2-k},\\
&&n_1=n_1(k,r)=1+\max\{\lc\frac{r-1}{\alpha}\rc,n_0+\binom{n_0}{2},16(k+1)^3r^8+6(k+1)^2r^3m_1\}.
\end{eqnarray*}

Now suppose that $G$ is a graph on $n>n_1$ vertices, with $\phi(G,\Fkr)\geqs\ex(n,\Fkr)$ and   $G$ is not an extremal graph for $\Fkr$. Then $e(G)>\ex(n,\Fkr)$. By Lemma~\ref{Lemma:deltaGgeq}, we may assume that $\delta(G)\ge\dtaT$. Note that $\chi(F_{k,r})=r$. Let $V_1,...,V_{r-1}$ be a partition of $V(G)$ such that $\ecrossG$ is maximized. Let $m=\einG$.  By  Lemma \ref{LEM: partition} and the choice of the partition of $V(G)$, we have $m<\gm n^2$. By Lemma~\ref{THM:Fkr}, observe that
$$m=e(G)-\ecrossG>\ex(n,\Fkr)-e(T_{n,r-1})=g(k),$$
and that
$$e(G)=m+\ecrossG\leqs m+e(T_{r-1,n})=\ex(n,\Fkr)+m-g(k).$$
So to prove Theorem \ref{THM: Main}, it suffices to show that
$$p_{\Fkr}(G)> \frac{m-g(k)}{e(\Fkr)-1}\ (\ge \frac{e(G)-\ex(n,\Fkr)}{e(\Fkr)-1}).$$

The following claim asserts that the partition $V(G)=V_1\dot{\cup}V_2\dot{\cup}\cdots\dot{\cup}V_{r-1}$ is very close to balance.

\begin{claim}\label{Fact:balance}
For every $i\in [1,r-1]$, $||V_i|-\aver|\leqs\error$.
\end{claim}
\begin{proof} Let $a=\max\{\left||V_j|-\aver\right|,~j\in[1,r-1]$. Wlog, suppose that $|V_1|-\aver=a$, then
\begin{eqnarray*}
e(G)&=&\ecrossG+m\\
&\leqs&\ecrossT+m\\
&=&|V_1|(n-|V_1|)+\ecrossTT+m\\
&=&|V_1|(n-|V_1|)+\frac12\left[\left(\sum_{j=2}^{r-1}|V_j|\right)^2-\sum_{j=2}^{r-1}|V_j|^2\right]+m\\
&\leqs& |V_1|(n-|V_1|)+\frac12(n-|V_1|)^2-\frac{1}{2(r-2)}(n-|V_1|)^2+m\\
&<&-\frac{r-1}{2(r-2)}a^2+\frac{r-2}{2(r-1)}n^2+\gm n^2.
\end{eqnarray*}
The fifth inequality holds by Jensen's inequality and the last inequality holds since $|V_1|=a+\aver$, and $m<\gm n^2$.

While on the other hand,
\begin{eqnarray*}
e(G)&>&ex(n,F_{k,r})> e(T_{n,r-1})\\
&\geqs&\binom{r-1}{2}\left(\aver-1\right)^2\\
&\geqs &\binom{r-1}{2}(\frac{\alpha n}{r-1})^2\\
&=&\frac{r-2}{2(r-1)}n^2\alpha^2.
\end{eqnarray*}
Therefore, $$\frac{r-1}{2(r-2)}a^2<\left[\frac{r-2}{2(r-1)}(1-\alpha^2)+\gm\right ]n^2<2\gmnn,$$ which implies that
$a<\sqrt{\frac{4(r-2)}{r-1}\gm}n<2\sgmn$.
\end{proof}

Recall that our purpose is to find more than $(m-g(k))/(e(\Fkr)-1)$ edge-disjoint copies of $\Fkr$. We continue the proof by considering two different cases.

\begin{case}$m>{m_1}.$
\end{case}
Set $t_1=\frac{n}{16kr^3}$ and $$t_2=\aver-2(r-3)\sgmn-2t_1-2.$$
For $v\in{V_i}$ ($i\in [1, r-1]$), we say that $v$ is $bad$ if $\deg_{G[V_i]}(v)>t_1$, otherwise $v$ is said to be $good$. For each bad vertex $v\in{V_i}$, choose $k\lceil{\deg_{G[V_i]}(v)/2k}\rceil$ edges which connect $v$ to good vertices in $V_i$. We keep these edges and delete other edges in $G[V_i]$ at $v$. We repeat this procedure to every bad vertex in $G$ and denote the resulting graph by $G_0$. This is possible since the number of bad vertices in $G$ is at most $$\frac{2m}{t_1}<\frac{2\gmnn}{t_1}=\bad<\frac{t_1}2.$$
Next we will find copies of $\Fkr$ in $G_0$. Each time we find a copy of $\Fkr$, we delete the edges of $\Fkr$, and let $G_s$ denote the graph obtained from $G_0$ after deleting the edges of $s$ copies of $\Fkr$.  For a vertex $v\in{V_i}$ in $G_s$, we say that $v$ is $active$ (in $G_s$) if $e_{G_s}(v, V_j)>t_2$ for every $j\neq{i}$, otherwise $v$ is said to be $inactive$.

We have the following two claims about $G_0$.

\begin{claim}\label{Fact:m(G_0)}
$$\sum_{i=1}^{r-1}e_{G_0}(V_i)\geqs{\frac{m}{2}}.$$
\end{claim}
\begin{proof}For every $i$, let $U_i\subset{V_i}$ be the set of good vertices in $V_i$. Then
\begin{eqnarray*}
e_{G_0}(V_i)&=&e_{G_0}(U_i)+\sum_{\begin{subarray}{c}v\in{V_i}\\{v \mbox{\ is bad}}\end{subarray}}\deg_{G_0[V_i]}(v)\\
&=&e_G(U_i)+\sum_{\begin{subarray}{c}v\in{V_i}\\{v \mbox{\ is bad}}\end{subarray}}k\lc{\frac{1}{2k}\deg_{G[V_i]}(v)}\rc\\
&\geqs&\frac12e_G(U_i)+\frac12\sum_{\begin{subarray}{c}v\in{V_i}\\{v \mbox{\ is bad}}\end{subarray}}\deg_{G[V_i]}(v)\\
&\geqs&\frac12e_G(V_i).
\end{eqnarray*}
\end{proof}

\begin{claim}\label{Fact:good is active} All of good vertices are active in $G_0$.
\end{claim}
\begin{proof}
Let $v\in{V_i}$ be a good vertex. Then for every $j\neq{i}$, by Claim~\ref{Fact:balance} and Lemma~\ref{Lemma:deltaGgeq},
\begin{eqnarray*}
e_{G_0}(v,V_j)&\geqs&\delta(G)-\deg_{G_0[V_i]}(v)-(r-3)(\upbound)\\
&\ge&\delta(T_{n,r-1})-\deg_{G[V_i]}(v)-(r-3)(\upbound)\\
&\geqs&\aver-2(r-3)\sgmn-t_1-1\\
&\geqs&t_2+t_1.
\end{eqnarray*}
\end{proof}

The following two steps are the procedures to find edge-disjoint copies of $\Fkr$.

\textbf{Step 1.} Suppose that we have gotten $G_s$ for some integer $s\geqs0$. If there exists a vertex $u\in{V_i}$ with $\deg_{G_s[V_i]}(u)\geqs{k}$ (bad vertices are considered first, followed by good vertices), let $v_i^1,...v_i^k$ be $k$ neighbors of $u$ in $V_i$, then for every $j\neq{i}$ we find $k$ good and active vertices $v_j^1,...,v_j^k\in{V_j}$ such that for every $\ell\in [1,k]$, $G_s[u,v_1^\ell,...,v_{r-1}^\ell]$ is a copy of $K_r$, and thus $G_s[\{u\}\cup\{v_j^\ell:\ell\in [1,k], j\in [1,r-1]\}]$ contains a copy of $\Fkr$ centered at $u$. Let $G_{s+1}$ be the updated new graph obtained from $G_s$ by deleting the edges of this $\Fkr$.

\textbf{Step 2.} After Step 1 is completed, denote the remaining graph by $G_a$. Then $\Delta(G_a[V_i])<k$ for each $i\in[1,r-1]$, and $\deg_{G_a[V_i]}(u)=0$ for all bad vertices $u\in{V_i}$ since $\deg_{G_0[V_i]}(u)$ is a multiple of $k$. Suppose that we have get $G_s$ for some $s\geqs{a}$. If there is a matching of order $k$ in $G_s[V_i]$ for some $i\in[1,r-1]$, for example, let $v_1^1w_1^1,...,v_1^kw_1^k\in{V_1}$ be such a matching. We find good and active vertices $u\in{V_2}$ and $v_j^1,...,v_j^k$ in $V_j$ for every $j\in [3,r-1]$ such that $G_s[v_1^\ell,w_1^\ell, u, v_3^\ell,\cdots, v_{r-1}^\ell]$ is a copy of $K_r$ for each $\ell\in [1,k]$ and so $G_s[\{u\}\cup\{v_1^\ell,w_1^\ell,v_j^\ell : j\in [3,r-1],\ell\in[1,k]\}]$ contains a copy of $\Fkr$ centered at $u$. Let $G_{s+1}$ be the updated new graph obtained from $G_s$ by deleting the edges of $\Fkr$. When Step 2 is completed, denote the remaining graph by $G_b$.

Note that after Step 1 and 2 are finished, we have found at least
$$\left\lfloor\frac1k\left(\sum_{i}e_{G_0}(V_i)-\sum_{i}e_{G_b}(V_i))\right)\right\rfloor$$
edge-disjoint $\Fkr$ from $G$ since each copy of $\Fkr$ using exactly $k$ edges from $\cup_{i=1}^{r-1}E(G[V_i])$.
Since $\Delta(G_b[V_i])\le k-1$ and $\nu(G_b[V_i])\le k-1$ for $i=1,2,\cdots, r-1$, by Lemma \ref{NuDelta}, $$e_{G_b}(V_i)\leqs{f(k-1,k-1)}\leqs{k(k-1)},$$
 which implies that $\sum_{i}e_{G_b}(V_i)\leqs{(r-1)k(k-1)}$. Let $p$ be the the number of removed edge-disjoint $\Fkr$ from $G$.

If $m>{m_2(k,r)}$ ($>m_1$), then, since $\sum_{i=1}^{r-1}e_{G_0}(V_i)\geqs{\frac{m}{2}}$, we have
$$p_{\Fkr}(G)\ge p\geqs\frac1k\left[\frac{m}{2}-(r-1)k(k-1)\right]>\frac{m-g(k)}{e(\Fkr)-1}.$$

If $m_1<m\leqs{m_2(k,r)}$, then for every vertex $u\in{V_i}$, $\deg_{G[V_i]}(u)\leqs{m}\leqs{m_2(k,r)}<t_1$ and hence $u$ is good. That is $G$ has no bad vertices and so $G_0=G$. Therefore,
$$p_{\Fkr}(G)\ge p\geqs\frac1k\left[m-(r-1)k(k-1)\right]>\frac{m-g(k)}{e(\Fkr)-1}.$$

Therefore, to complete the proof of Case 1, it remains to prove the following claim.

\begin{claim}\label{CLAIM: 5}
Step 1 and 2 can be successfully iterated.
\end{claim}

To proof the above claim, we first estimate the number of good and inactive vertices in each iteration of Step 1 or 2.
\begin{claim}\label{CORO:good inactive vetices in Gs}
Let $G_s\subset{G_0}$ be a subgraph at some point of the iteration in Step 1 or Step 2. Then the number of good inactive vertices in $G_s$ is at most $\inactive$.
\end{claim}
\begin{proof}
Since in each iteration of Step 1 or Step 2, the number of removed edges with both endpoints in $V_i$ is exactly $k$, $s\leqs{m/k}<\gmnn/k$.   So the total number of deleted edges from $G_0$ is at most $e(\Fkr)\cdot s<\frac{r(r-1)}2\gmnn$. By Claim \ref{Fact:good is active}, for every good vertex $u\in{V_i}$, $u$ is active in $G_0$ and $e_{G_0}(v,V_j)\geqs t_2+t_1$. Thus the number of good vertices that are inactive in $G_s$ is at most
$\frac{e(\Fkr)\cdot s}{t_1}<8kr^5\gm n.$
\end{proof}

\noindent {\sl{Proof of Claim~\ref{CLAIM: 5}:}} Let $G_s$ be the graph obtained at some point of the iteration in Step 1. For any fixed $x\in{V_i}$ and every $j\neq{i}$.  If $x$ is bad and is involved in a previous iterate, then $x$ is the center of a copy of $\Fkr$ and hence the number of removed edges that $x$ sends to $V_j$ is $k$. Since $x$ involves at most $\frac{\deg_{G_0[V_i]}(x)}k$ iterates,
\begin{eqnarray*}
e_{G_s}(x, V_j)&\geqs& e_{G_0}(x, V_j)-k\cdot \frac{\deg_{G_0[V_i]}(x)}k\\
&=&e_{G}(x, V_j)-k\lc{\frac{\deg_{G[V_i]}(x)}{2k}}\rc\\
&\geqs&\frac{ e_G(x,V_j)}2-k\\
&\geqs&\frac{n}{4(r-1)}-\frac{r-3}{2}\sgmn-\frac14-k\\
&\ge & \frac {t_2}4+\frac {t_1}2+\frac14-k,
\end{eqnarray*}
the third inequality holds since
\begin{equation}\label{EQU: e1}
e_G(x, V_j)\ge \deg_{G[V_i]}(x)
\end{equation}
((\ref{EQU: e1}) holds because of the maximality of $\ecrossG$); the forth inequality holds since
\begin{eqnarray}\label{EQN: e1}
e_G(x, V_j)&=&\deg_G(x)-\deg_{G[V_i]}(x)-\sum_{\ell\not=i,j}e_G(x, V_\ell)\\
          &\ge& \delta(G)-\deg_{G[V_i]}(x)-(r-3)(\upbound)\nonumber
\end{eqnarray}
and ((\ref{EQU: e1}) and (\ref{EQN: e1}) implies that)
\begin{eqnarray*}
e_G(x, V_j)
&\geqs&\frac12[\delta(G)-(r-3)(\upbound)] \\
&\ge& \frac{n}{2(r-1)}-(r-3)\sgmn-\frac  12.
\end{eqnarray*}

If $x$ is good and active in $G_s$, then $e_{G_s}(x,V_j) \geq t_2$. Now suppose $x$ is good but inactive in $G_s$. Then $x$ becomes inactive in a previous iterate. If $x$ is involved in a succeeding iterate, then $x$ is chosen to be the center of a copy of $\Fkr$. So the number of edges that $x$ sends to $V_j$ is $k$ and $x$ is involved in at most $\frac{\deg_{G_0[V_i]}(x)}k$ succeeding iterates. Hence, by inequality~(\ref{EQN: e1}) and $\deg_{G[V_i]}(x)\le t_1$,
\begin{eqnarray*}
e_{G_s}(x,V_j)& \geqs& t_2-k-\deg_{G_0[V_i]}(x)\\
              &\ge & t_2-t_1-k\\
(&\ge & \frac {t_2}4+\frac {t_1}2+\frac14-k\ \  \mbox {when $n$ is sufficiently large}).
\end{eqnarray*}

Wlog, let $u$ be in ${V_1}$ and $v_1^1,...,v_1^k\in{V_1}$ are $k$ (good) neighbors of $u$. Suppose that for some $\ell\in [1, r-2]$, we have found good and active vertices $v_i^1,...,v_i^k\in{V_i}$ \ ($1\leqs{i}\leqs{l}$) such that for every $j\in [1,k]$, $G_s[u,v_1^j,...,v_\ell^j]$ is a copy of $K_{\ell+1}$. Then, for every $j\in [1,k]$, by Claim~\ref{CORO:good inactive vetices in Gs}, the number of good and active common neighbors of $u,v_1^j,...,v_\ell^j$ in $V_{\ell+1}$ of $G_s$, denoted by $L_{\ell+1}(u,v_1^j,...,v_\ell^j)$, is
\begin{eqnarray*}
|L_{\ell+1}(u,v_1^j,...,v_\ell^j)|&\ge&e_{G_s}(u, V_{\ell+1})+\sum_{i=1}^\ell{e_{G_s}(v_{i}^j, V_{\ell+1})}-\ell|V_{\ell+1}|\\
                                  &&-\bad-\inactive\\
                   & \ge & \frac {t_2}4+\frac {t_1}2+\frac14-k+\ell(t_2-t_1-k)\\
                   & &-\ell(\upbound)-\bad-\inactive\\
&\geqs&\frac{n}{4(r-1)}-\frac{13n}{40kr^2}-(k+2)r-k-\frac 14\\
&\geqs&\frac{n}{4r}\ \  \mbox{(when $n$ is sufficiently large)}.
\end{eqnarray*}
Therefore, Step 1 can be performed successfully to find a copy of $\Fkr$ centered at $u$.

Now, let $G_s$ be the graph obtained at some point of the iteration in Step 2 for some $s\geqs{a}$. Let $x\in{V_i}$ be a good vertex and $j\neq{i}$. Then, after $x$ becomes inactive in previous iterates, the total number of removed edges that $x$ sends to $V_j$ are at most $\deg_{G_a[V_i]}(x)$. Hence
\begin{eqnarray*}
e_{G_s}(x, V_j)&\geqs&t_2-2k-\deg_{G_a[V_i]}(x) \\
&\geqs&t_2-2k-\deg_{G[V_i]}(x)\\
&\geqs&t_2-2k-t_1.
\end{eqnarray*}

Wlog, suppose that $v_1^1w_1^1,...,v_1^kw_1^k$ is a matching in $V_1$. Then $v_1^1,w_1^1,...,v_1^k,w_1^k$ are good. Let $X=\{v_1^1,w_1^1,...,v_1^k,w_1^k\}$. By CLaim~\ref{CORO:good inactive vetices in Gs}, the number of common active good neighbors of $X$ in $V_2$ of $G_s$, denoted by $L_2(X)$, is at least
\begin{eqnarray*}
|L_2(X)|&= &\sum_{x\in X}e_{G_s}(x,V_2)-(2k-1)|V_2|-\bad-\inactive\\
&\geqs&2k(t_2-2k-t_1)-(2k-1)(\upbound)-\bad-\inactive\\
&=&\frac{n}{r-1}-2[2k(r-2)-1]\sgmn-6kt_1-4k(k+1)-\bad-\inactive\\
&\geqs&\frac{n}{r-1}-\frac{n}{2r^2}-4k(k+1)\\
&\geq& \frac nr \ \ (\mbox{when $n$ is sufficiently large}).
\end{eqnarray*}
Choose such a common neighbor $u$ of $X$ in $V_2$.

Now, suppose that we have found active and good vertices $v_i^1,...,v_i^k$, for $i\in [3,\ell]$ ($2\leqs{\ell}\leqs{r-2}$) such that for every $j\in[1,k]$, $G_s[v_1^j,w_1^j,u,v_3^j,...,v_\ell^j]$ is a copy of $K_{\ell+1}$. Let $Y_j=\{v_1^j,w_1^j,u,v_3^j,...,v_\ell^j\}$ for $j\in [1,k]$ and denote the the common active and good neighbors of $Y_j$ in $V_{\ell+1}$ of $G_s$ by $L_{\ell+1}(Y_{j})$. Then the same reason as before,
\begin{eqnarray*}
|L_{\ell+1}(Y_j)|&\ge & \sum_{x\in Y_j}e_{G_s}(x,V_{\ell+1})-\ell\cdot|V_{\ell+1}|-\bad-\inactive\\
            &\geqs&(\ell+1)(t_2-2k-t_1)-\ell(\upbound)-\bad-\inactive\\
            &=&\aver-2[(r-2)(\ell+1)-1]\sgmn-3(\ell+1)t_1-\bad\\
            & & -\inactive-2(\ell+1)(k+1)\\
&\geqs&\aver-2((r-2)(r-1)-1)\sgmn-3(r-1)t_1-\bad\\
&& -\inactive-2(r-1)(k+1)\\
&\geqs&\frac{n}{r-1}-\frac{3n}{8kr^2}-2(r-1)(k+1)\\
&\ge & \frac{n}r\ \ (\mbox{when $n$ is sufficiently large}).
\end{eqnarray*}
Hence $k$ different active and good common neighbors $v_{\ell+1}^1, \cdots, v_{\ell+1}^k$ with $v_{\ell+1}^j\in L_{\ell+1}(Y_{j})$ for $j\in [1,k]$ do exist and therefore Step 2 can be successfully iterated.

\vspace{5pt}

\begin{case} $m\leqs{m_1}.$
\end{case}
For every $i\in [1,r-1]$, let $B_i\subset{V_i}$ be the set of isolated vertices of $G[V_i]$ and $A_i=V_i\setminus{B_i}$. Then $|A_i|\leqs{2m}\leqs{2m_1}$. Since $|V_i|\geqs\lowbound>2m_1\ge |A_i|$, we have $B_i\neq\emptyset.$
Note that, for $u\in{B_i}$, $\deg_{G[V_i]}(u)=0$. By $\dtaT\leqs\deg_G(u)\leqs{n-|V_i|}$, we have $|V_i|\leqs{\lc{\aver}\rc}$.
Together with $|V_1|+...+|V_{r-1}|=n$, we have the following claim.

\begin{claim}\label{Fact:r-2 balance} For every  $i\in[1,r-1]$, $\lc{\aver}\rc-(r-2)\leqs|V_i|\leqs{\lc{\aver}\rc}$. Particularly, if $(r-1)|n$, then $|V_i|=\aver$ for each $i\in [1,r-1]$.
\end{claim}

Since $e(G)\geqs\phi(G,\Fkr)\geqs{\ex(n,\Fkr)}$, there exists some integer $s\geqs0$, such that
$$s(e(\Fkr)-1)\leqs{e(G)-\ex(n,\Fkr)}<(s+1)(e(\Fkr)-1).$$
Note that $e(G)-\ex(n,\Fkr)\leqs{m-g(k)}\leqs{m_1-g(k)}$. Hence $s\leqs\frac{m_1-g(k)}{e(\Fkr)-1}.$
Furthermore, we have a simple and useful upper bound for $s$ as follows.

\begin{claim}\label{Fact:upbound for s}
 $s<\frac{(r-2)(k-1)}2+1.$ Particularly, $s<e(\Fkr)$.
\end{claim}
\begin{proof} For convenience, write $e=e(\Fkr)=\frac{kr(r-1)}2$  and $g=g(k)$. Since $s\leqs\frac{m_1-g}{e-1}$, it's sufficient to prove $1+\frac{(r-2)(k-1)}2>\frac{m_1-g}{e-1}$. Note that $m_1=\frac{(e-1)(r-1)k(k-1)-kg}{e-k-1}$. So
\begin{eqnarray*}
&&\frac{(r-2)(k-1)}2+1>\frac{m_1-g}{e-1}\\&\Leftrightarrow &\frac{(r-2)(k-1)}{2}+1>\frac{(r-1)k(k-1)-g}{e-k-1}\\
                                      &\Leftrightarrow & \frac{(r-2)(k-1)}{2}+1>\frac{(r-1)k(k-1)-k(k-3/2)}{e-k-1} \ (\mbox{since $g\ge k(k-3/2)$})\\
                                      &\Leftrightarrow & \frac{(r-2)(k-1)}{2}+1>\frac{(r-2)k(k-1)+k/2}{e-k-1}\\
                                      &\Leftrightarrow & [(r-2)(k-1)+2](e-k-1)-2(r-2)k(k-1)-k>0\\
                                      &\Leftrightarrow & (r-2)(k-1)(e-3k-1)+2e-3k-2>0
\end{eqnarray*}
Note that $e=\frac{kr(r-1)}2\ge 3k$ ($r\ge 3$) and the equality holds if and only if $r=3$. It is an easy task to check that the above inequality always holds for $k\ge 1$.

Clearly, $\frac{kr(r-1)}2>\frac{(r-2)(k-1)}{2}+1$ when $r\ge 3$ and $k\ge 1$. So $s<e(\Fkr)$.
\end{proof}

If we can find $s+1$ edge-disjoint copies of $\Fkr$ in $G$, then we have $$\phi(G,\Fkr)\leqs{e(G)-(s+1)(e(\Fkr)-1)}<\ex(n,\Fkr),$$
a contradiction with the assumption that $\phi(G,\Fkr)\ge \ex(n,\Fkr)$. So to complete the proof of Case 2 (and the proof of Theorem \ref{THM: Main}), it remains to prove the following claim.

\begin{claim}\label{s+1}
 $G$ contains $s+1$ edge-disjoint copies of $\Fkr$.
\end{claim}

Before we prove Claim \ref{s+1}, we need an auxiliary claim.

\begin{claim}\label{Claim:(r-1)complete}
For every $i\in [1,r-1]$, and any subset $A^i\subset{A_i}$, if $A^i\neq\emptyset$ and $|A^i|\leqs(k+1)(s+1)$, then for every $j\neq{i}$, there exists a subset $B^j\subset{B_j}$ with $|B^j|=(k+1)(s+1)$ such that $G$ contains a complete $(r-1)$-partite subgraph of $G$ with partitions $A^i$, $B^j$ ($j\in [1, r-1]$ and $j\neq{i}$).
\end{claim}
\begin{proof} Wlog, let $A^1\subset{A_1}$ and $|A^1|=(k+1)(s+1)$. Suppose that we have found $B^2,...,B^\ell$, $B^j\subset{B_j}$ with $|B^j|=(k+1)(s+1)$, $j\in [2,\ell]$, $\ell\in [1,r-2]$, such that $G$ contains a complete $\ell$-partite subgraph with partitions $A^1$ and $B^2,...,B^\ell$.
Note that $|V_i|\le \lceil\frac n{r-1}\rceil$ ($i\in [1, r-1]$) by Claim~\ref{Fact:r-2 balance}. Hence for any $u\in{A^1}$ and $v\in{B^j}$,
\begin{eqnarray*}
e_G(u, V_{\ell+1})&\geqs&\delta(G)-(r-3)(\aver+1)-|A_1|\\
&\geqs&\aver-(r-2)-2m_1,\\
e_G(v,V_{\ell+1})&\geqs&\delta(G)-(r-3)(\aver+1)\\
&\geqs&\aver-(r-2)-2m_1.
\end{eqnarray*}

Let $Z^{\ell}=A^1\cup B^2\cup...\cup B^\ell$ and let $L_{\ell+1}(Z^\ell)$ be the set of common neighbors of $Z^{\ell}$ in $B_{\ell+1}$. Then
\begin{eqnarray*}
|L_{\ell+1}(Z^\ell)| &\ge&\sum_{v\in{Z^{\ell}}} e_G(v, V_{\ell+1})-[(k+1)(s+1)\ell-1]|V_{\ell+1}|-|A_{\ell+1}|\\
&\geqs&(k+1)(s+1)\ell[\aver-(r-2)-2m_1]\\
&&-[(k+1)(s+1)\ell-1](\aver+1)-2m_1\\
&=&\aver-(2m_1+r-1)(k+1)(s+1)\ell-2m_1+1\\
&\geqs&\frac{n}{r}\ \ (\mbox{when $n$ is sufficiently large}).
\end{eqnarray*}
Hence we always can choose $B^{\ell+1}\subset B_{\ell+1}$ for $\ell\in[1, r-2]$ and so the result follows.
\end{proof}

\vspace{5pt}
\noindent {\sl{Proof of Claim~\ref{s+1}:}}
By Claim~\ref{Fact:upbound for s},  $s<e(\Fkr)$. By $e(G)\geqs{\ex(n,\Fkr)+s(e(\Fkr)-1)}$, we have
$$e(G)-(s-1)e(\Fkr)\geqs{\ex(n,\Fkr)+e(\Fkr)-s>\ex(n,\Fkr)}.$$
So, $G$ contains $s$ edge-disjoint copies of $\Fkr$. Let $G'$ be a subgraph of $G$ by removing the edges of $s$ copies of $\Fkr$. If there is a vertex $u$ in $G'$ with
$\deg_{G'}(u)\leqs\dtaT-s,$
then
\begin{eqnarray*}
e(G'-u)&=&e(G')-\deg_{G'}(u)\\
&=&e(G)-s\cdot e(\Fkr)-\deg_{G'}(u)\\
&\geqs&\ex(n,\Fkr)-\dtaT\\
&=&\ex(n-1,\Fkr).
\end{eqnarray*}
 If the equality does not hold or the equality holds but we can show that $G'-u$ is not an extremal graph for $\Fkr$, then $G'-u$ contains a copy of $\Fkr$ and we are done.



In the following, we show that how to remove the edges of $s$ copies of $\Fkr$ in $G$ to get our desired subgraph $G'$ and vertex $u$ according to three cases.

\textbf{(I).} If there are $s$ vertices $u_1,...,u_s\in{A_1\cup...\cup A_{r-1}}$ such that $\deg_{G[A_i]}(u_j)\geqs{k+s-1}$ for $u_j\in{A_i}$. Let $q=\max\{|\{u_1,...,u_s\}\cap{A_i}| :  i\in [1,r-1]\}$. Then $s\geqs{q}\geqs\lc{\frac{s}{r-1}}\rc$. Wlog, assume that  $u_1,...,u_q\in{A_1}$. For each $j\in [1,q]$, choose $k$ neighbors of $u_j$ in $A_1\setminus\{u_1,...,u_q\}$, say $v_{j1},...,v_{jk}$, this is possible since there are at least $k+s-1-(q-1)=k+s-q\geqs{k}$ such neighbors. We further assume that $v_{j,1}\neq v_{\ell,1}$ for $1\leqs{j}<\ell\leqs{q}$. Let $A_j^1=\{u_j,v_{j1},...,v_{jk}\}$ for $j\in [1,q]$. Then $|\cup_{j=1}^q A_j^1|\le q(k+1)$. By Claim \ref{Claim:(r-1)complete}, we can find a common neighbor $u\in B_{i_0}$ of $\cup_{j=1}^q A_j^1$ for some $i_0\neq1$, here $i_0$ will be determined later. For each $j\in [1,q]$, since $|A_j^1|=k+1$, by Claim \ref{Claim:(r-1)complete}, we can choose $B_j^i\subset{B_i}$ for $i\in [2, r-1]$ such that
\begin{eqnarray*}
&(a)&|B_j^i|=k,\\
&(b)&B_j^{i_0}\cap{B_\ell^{i_0}}=\{u\}\  \mbox{for $j, \ell\in [1,q]$ and ${j}\not=\ell$},\\
&(c)&B_j^i\cap{B_\ell^i}=\emptyset\ \mbox{ for  $i\geqs2$, $i\neq{i_0}$ and ${j}\not=\ell$},\\
&(d)& \mbox{$G$ contains a complete $(r-1)$-partite graph with partitions $A_j^1, B^2_j,...,B^{r-1}_j$}.
\end{eqnarray*}
Assume $B_j^i=\{u_{j1}^i,...,u_{jk}^i\}$ and $u_{j1}^{i_0}=u$ for every $j\in [1,q]$ and $i\in [2, r-1]$. Then for every $j\in [1,q]$ and every $\ell\in[1,k]$, $G_{j\ell}= G[u_j,v_{j\ell},u^2_{j\ell},...,u^{r-1}_{j\ell}]$ is a copy of $K_r$ and thus $G_j= G_{j1}\cup...\cup G_{jk}$ contains a copy of $\Fkr$ centered at $u_j$. Furthermore, $G_1,...,G_q$ are $q$ edge-disjoint copies of $\Fkr$ with a common vertex $u$. Let $G'$ be the subgraph obtained from $G$ by deleting the edges of $G_1, G_2, \ldots, G_q$ and any other $s-q$ copies of $\Fkr$ in $G$.

(I.1). If $q\geqs\lc{\frac{s}{r-1}}\rc+1$, then
 $$\deg_{G'}(u)\leqs n-|V_{i_0}|-(r-1)q\leqs \lf{\frac{r-2}{r-1}n}\rf-s-1<\lf{\frac{r-2}{r-1}n}\rf-s,$$
 the second inequality holds since $|V_{i_0}|\ge \lc{\frac{n}{r-1}}\rc-(r-2)$ by Claim~\ref{Fact:r-2 balance}. Hence $u$ is the desired vertex and we are done.

(I.2). $q=\lc{\frac{s}{r-1}}\rc$, let $s\equiv{t}\mod(r-1)$, ($t\in [0,r-2]$).

 If $t=0$ or $t\geqs2$, then at least two subsets of $A_1,\ldots, A_{r-1}$ satisfying that $|\{u_1,\ldots, u_s\}\cap A_i|=q$ by the choice of $q$. By Claim~\ref{Fact:r-2 balance}, at least one of $V_i$, $i\in [1,r-1]$ with $|V_i|=\lc{\frac{n}{r-1}}\rc$. Hence we can choose $V_1$ and $i_0\neq1$ such that $|V_1|\le \lc{\aver}\rc$ and $|V_{i_0}|=\lc{\aver}\rc$. So,
$$\deg_{G'}(u)\leqs n-\lc{\aver}\rc-(r-1)q\leqs \lf{\frac{r-2}{r-1}n}\rf-s.$$
Here we need to show that $G'-u$ is not an extremal graph for $\Fkr$. Let $H=G'-u$.
Consider $u_{12}^{i_0}\in{V(H)}$. Then $\deg_H(u_{12}^{i_0})\le n-\lc\aver\rc-(r-1)<\lf\frac{r-2}{r-1}(n-1)\rf$. By Lemma~\ref{PROP: ex-p1}, $H$ is not an extremal graph on $n-1$ vertices for $\Fkr$ and hence we are done.

If $t=1$, then $q=\frac{s-1}{r-1}+1$. Choose $i_0$ such that $|V_{i_0}|=\max\limits_{\ell\in[2,r-1]}|V_\ell|$. Then $|V_{i_0}|\geqs\lf\aver\rf$, otherwise, by Claim~\ref{Fact:r-2 balance}, $|V_1|+...+|V_{r-1}|\leqs\lc\aver\rc+(r-2)(\lf\aver\rf-1)<n.$ Hence, by Claim~\ref{Fact:r-2 balance},
$$\deg_{G'}(u)\leqs n-\lf{\aver}\rf-(r-1)q=\lf{\frac{r-2}{r-1}n}\rf-s-(r-1)<\lf{\frac{r-2}{r-1}n}\rf-s.$$
So we are done in this case.

\textbf{(II).} Suppose that $G$ contains a copy of $F_{qk,r}$ with center $u\in{A_1\cup...\cup A_{r-1}}$ and $q\geqs 2$. We choose $u$ so that $q$ is maximum. If $q\geqs{s+1}$ then we are done since $F_{qk,r}$ contains $s+1$ edge-disjoint copies of $\Fkr$. So $q\leqs{s}$. Then we must have $\deg_G(u)<n-\left(\lc\aver\rc-(r-2)\right)+(q+1)k$, otherwise, $\deg_{G[A_i]}(u)\ge (q+1)k$ for some $i\in [1, r-1]$ with $u\in A_i$, then Claim \ref{Claim:(r-1)complete} guarantees that  there exists a copy of $F_{(q+1)k,r}$  centered at $u$ in $G$, contradicting the choice of $q$. Let $G'$ be the graph obtained from $G$ by deleting the edges of the copy of $F_{qk,r}$ and the edges of any further $s-q$ copies of $\Fkr$. We have
\begin{eqnarray*}
\deg_{G'}(u)&\leqs&n-\left(\lc\aver\rc-(r-2)\right)+(q+1)k-1-(r-1)qk\\
&=&\dtaT+(r-3)-((r-2)q-1)k\\
&\leqs&\dtaT+(r-3)-(2r-5)k \ \ \mbox{(since $q\ge 2$)}\\
&=& \dtaT-\left[\frac{(r-2)(k-1)}2+1\right]-\frac{(r-2)(k-1)}2-(r-3)k\\
&<&\dtaT-s  \ \ \mbox{(since $s<\frac{(r-2)(k-1)}2+1$ by Claim~\ref{Fact:upbound for s})}.
\end{eqnarray*}
Hence $u$ is the desired vertex and we are done.

\vspace{5pt}
\textbf{(III).} Now, suppose that \textbf{(I)} and \textbf{(II)} do not hold. We obtain $G'$ from $G$ by deleting the edges of any $s$ copies of $\Fkr$. If there is a copy centered at $u\in{B_1\cup...\cup B_{r-1}}$, then,  by Claim~\ref{Fact:r-2 balance},
\begin{eqnarray*}
\deg_{G'}(u)&\leqs&n-\left(\lc\aver\rc-(r-2)\right)-(r-1)k\\
&=&\dtaT-(r-1)(k-1)-1\\
&<&\dtaT-s\ \ \mbox{(since $s<\frac{(r-2)(k-1)}2+1$ by Claim~\ref{Fact:upbound for s})}.
\end{eqnarray*}
Hence, all the centers of the $s$ copies of $\Fkr$ lie in $A_1\cup...\cup A_{r-1}$. By \textbf{(II)}, they must be pairwisely distinct. By \textbf{(I)}, at most $s-1$ of them have degree at least $k+s-1$ in $G[A_i]$, $i\in[1, r-1]$. Hence, at least one of the centers, say $u\in{A_i}$ with $\deg_{G[A_i]}(u)\leqs{k+s-2}$. By Claim~\ref{Fact:r-2 balance} and $s<\frac{(r-2)(k-1)}2+1$, we have
\begin{eqnarray*}
\deg_{G'}(u)&\leqs&n-\left(\lc\aver\rc-(r-2)\right)+k+s-2-(r-1)k\\
&=&\dtaT-(r-2)(k-1)+s-2\\
&<&\dtaT-s.
\end{eqnarray*}
Hence $u$ is a desired vertex in $G'$ and we are done.

This completes the proof of Claim~\ref{s+1} and also completes the proof of Theorem~\ref{THM: Main}.

\noindent{\bf Acknowledgement.} We would like to thank Professor Jie Ma for many helpful comments and suggestions.


\begin{thebibliography}{99}
\bibitem{ImprovedError}
P. Allen, J, B\"{o}ttcher, and Y. Person, An improved error term for minimum $H$-decompotions of graphs, {\it J. Combin. Theory Ser. B}, {\bf108} 92--101, 2014.

\bibitem{Bollobas}
B. Bollob\'{a}s, On complete subgraphs of different orders, {\it Math. Proc. Cambridge. Philos. Soc.} {\bf 79(1)} (1976) 19--24.


\bibitem{Hanson}
V. Chav\'atal, D. Hanson, Degrees and matchings, {\it J. Combin. Theory Ser.B}, {\bf 20(2)} (1976) 128--138.

\bibitem {Wei}
G. Chen, R. J. Gould, F. Pfender, and B. Wei, Extremal graphs for intersecting cliques, {\it J. Combin. Theory Ser. B} {\bf 89} (2003) 159--171.

\bibitem{intersections}
P. Erd\"{o}s, A. W. Goodman, L. P\'{o}sa, The representation of a graph by set intersections, {\it Canad. J. Math.} {\bf18}(1966) 106--112.

\bibitem{k-fan}
H. Liu, T. Sousa, Decompositions of Graphs into Fans and Single Edges, manuscript. 

\bibitem{edge-critical case}
L. \"{O}zkahya and Y. Person, Minimum $H$-decomposition of graphs: edge-critical case, {\it J. Combin. Theory Ser. B}, {\bf102(3)} 715--725, 2012.

\bibitem{general}
O. Pikhurko and T. Sousa, Minimum $H$-decompositions of graphs,  {\it J. Combin. Theory Ser.B}, {\bf 97(6)} (2007) 1041--1055.

\bibitem{5-cycles}
T. Sousa, Decompositions of graphs into 5-cycles and other small graphs, {\it Electron. J. Combin.,} 12:Research Paper 49, 7 pp. (electronic), 2005.

\bibitem{7-cycle}
T. Sousa, Decompositions of graphs into cycles of length seven and single edges, {\it Ars Combin.,} to appear.

\bibitem{clique-extension}
T. Sousa, Decompositions of graphs into a given clique-extension, {\it Ars Combin.,} {\bf100} 465--472, 2011.


\end{thebibliography}
\end{document}